\documentclass[letterpaper, 10 pt, conference]{ieeeconf}  

\IEEEoverridecommandlockouts                              

\usepackage{mathrsfs}
\usepackage{xcolor}

\usepackage{amsmath,amssymb,amsfonts}
\usepackage{graphicx}
\usepackage{algorithm,algorithmic}

\usepackage{textcomp}

\usepackage{setspace}
\usepackage{lipsum}

\usepackage{mathtools}
\usepackage{enumerate}

\usepackage{caption}
\usepackage{subcaption}
\usepackage{algorithm}
\usepackage{float}
\newtheorem{definition}{Definition}
\newtheorem{theorem}{Theorem}

\newtheorem{corollary}{Corollary}

\usepackage{hyperref}

\makeatletter
\let\NAT@parse\undefined
\makeatother
\usepackage{cite}

\newtheorem{examplex}{Example}
\newenvironment{example}
  {\begin{examplex}}
  {\hfill$\blacksquare$\end{examplex}}

\newcommand{\R}{\mathbb{R}}

\newcommand{\sgn}{\mathrm{sgn}}
\title{Inverse Optimal Feedback Stabilization of\\ 3D Nonholonomic Vehicles in Spherical Coordinates
}

\author{Kwang Hak Kim, Mamadou Diagne, and Miroslav Krsti\'c%
\thanks{This work was supported by the Office of Naval Research under grants N00014-23-1-2376 and N00014-23-1-2831. The results and opinions in this paper are solely of the authors and do not reflect the position or the policy of the U.S. Government.}
\thanks{K. H. Kim, M. Diagne, and M. Krsti\'c are with the Department of Mechanical and Aerospace Engineering, UC San Diego, 9500 Gilman Drive, La Jolla, CA, 92093-0411, {\tt\small \{kwk001,mdiagne,krstic\}@ucsd.edu}}%
}

\begin{document}

\maketitle


\begin{abstract}
Inverse optimal control provides optimality and robustness guarantees without solving the Hamilton--Jacobi--Bellman (HJB) equation, but its application to 3D nonholonomic vehicles has been precluded by the absence of strict control Lyapunov functions (CLFs). Recently, such strict CLFs were constructed for the 3D nonholonomic vehicle in spherical coordinates, which circumvent Brockett's obstruction, on the largest possible domain. Building on this construction, this paper develops a general inverse optimal stabilizing controller for the 3D nonholonomic vehicle actuated by surge velocity, pitch rate, and yaw rate. The design is a damping $L_gV$ feedback that minimizes a meaningful cost and inherits stability margins. Its flexibility is demonstrated through two choices of penalty function: a quadratic running cost on the state, recovering a near-classical optimal control formulation, and stabilization under user-defined input constraints. Numerical simulations illustrate stabilization performance, control-effort tradeoffs, and the effect of input constraints.
\end{abstract}


\section{Introduction}

Many underwater vehicles~\cite{lapierre2003nonlinear,caccia2000guidance}  and aerial vehicles~\cite{vamvoudakis2022nonequilibrium,lugo2014dubins} share a kinematic motion along the body's longitudinal axis and steering in pitch and yaw, but no direct lateral (sway) or vertical (heave) actuation. The model actuated in surge velocity, pitch rate, and yaw rate~\cite{fossen1994guidance} is the minimal representation of this class of systems. Most control results for this class address path following and tracking~\cite{aguiar2007trajectory,fossen2024alos3d}, where the moving reference supplies the excitation that point stabilization lacks. However, stabilization of nonholonomic vehicles is fundamentally more challenging~\cite{jiang2010controlling}, since no such persistent excitation exists.

In Cartesian/Euler coordinates, the necessary conditions of Brockett, extended by Ryan, Coron, and Rosier~\cite{brockett1983asymptotic,coron1994relation,ryan1994brockett}, exclude asymptotic stabilization by time-invariant state feedback, whether continuous or discontinuous. The 3D stabilization literature relies primarily on time-varying feedback~\cite{pettersen1999time,do2002global}, which exhibits slow, oscillatory convergence; discontinuous feedback~\cite{egeland1994exponential,egeland1996feedback}, which is prone to chattering; or logic-based hybrid strategies, which can yield unpredictable, unnatural maneuvers~\cite{aguiar2002global}. Notably, \cite{Heetal2022exp3D} achieves semi-global exponential stabilization on SE(3) via a discontinuous logarithmic feedback. However, the exponential guarantee hinges on a gain condition involving a state-dependent eigenvalue bound that, as the authors note, is difficult to verify a priori and may fail for large initial displacements, leaving convergence in that region supported only by simulation.

A singular change of coordinates offers a different route around the obstruction. For the unicycle, the canonical nonholonomic system, rewriting the kinematics in polar coordinates allows stabilization by continuous time-invariant feedback and, crucially, supports strict control Lyapunov functions (CLFs)~\cite{aicardi1995,Part1_todorovskiCLF2025}. The polar states, range and bearing to the target, are moreover the quantities that LiDAR, radar, and camera sensors measure directly, such that the framework operates on egocentric measurements without reconstruction of global states. Lyapunov-based designs in polar-like coordinates exist for the 3D vehicle, but rely on discontinuous feedback~\cite{aicardi2001cusp} or on piecewise-defined feedback with cascade arguments and negative-semidefinite Lyapunov derivatives~\cite{restrepo_3d_2019}. In either case, they guarantee only asymptotic, non-exponential convergence and yield no strict CLF, which rules out inverse optimal redesign~\cite{Part2_kimIOC2025}.

Optimal control systematically balances state convergence against control effort, and optimal feedback laws inherit desirable robustness margins. For nonlinear systems, however, direct optimal design requires solving the Hamilton--Jacobi--Bellman (HJB) equation, which is often intractable. Inverse optimal control circumvents this by constructing the feedback law from a strict CLF. The resulting controller is optimal with respect to a meaningful cost and retains the robustness margins of optimal designs without solving the HJB equation~\cite{sepulchre1997constructive,krstic1998inverse}. 
The benefits of inverse optimality naturally motivate its application to underactuated systems, where existing results are limited to practical stability~\cite{do2021global}, subsystem-level optimality in the stochastic setting~\cite{do2015global}, or the planar case~\cite{Part2_kimIOC2025}. The inverse optimal control design for 3D nonholonomic vehicles, however, has been precluded by the lack of strict CLFs.

Recently, global exponential stability and strict CLFs in spherical coordinates have been established for the 3D nonholonomic vehicle on the largest domain on which the representation is well defined~\cite{kim2026spherical3d}. Building on these strict CLFs and the polar-coordinate framework of~\cite{Part2_kimIOC2025}, this paper develops families of inverse optimal stabilizers for the 3D nonholonomic vehicle in the class of damping $L_gV$ controllers~\cite{jurdjevic1978controllability,sontag2013mathematical}. The resulting feedback is optimal with respect to a meaningful cost without solving the HJB equation, retains stability margins, and accommodates user-defined input bounds via the Legendre--Fenchel transform. Two examples illustrate this flexibility through different penalty choices, both demonstrated in simulation.

\section{3D Model in Spherical Coordinates}\label{sec:model}

Consider the 3D nonholonomic kinematic model in the North-East-Down (NED) frame~\cite{fossen1994guidance}:
\begin{subequations}\label{eq:cart_3d_sys}
\begin{align}
\dot x &= v \cos\psi \cos\theta\label{eq:x_dot}\\
\dot y &= v \sin\psi \cos\theta\label{eq:y_dot}\\
\dot z &= -v \sin\theta\\
\dot \theta &= q\\
\dot \psi &= \frac{r}{\cos\theta}\label{eq:psi_dot}\,
,
\end{align}
\end{subequations}
where $(x,y,z)\in\R^3$ is the position, $\psi \in \mathbb{R}$ is yaw, $\theta \in\mathbb{R}$ is pitch (positive nose-up), and $(v,q,r)$ are forward velocity (surge), pitch rate, and yaw rate inputs respectively. The singularity caused by the division of $\cos\theta$ in~\eqref{eq:psi_dot} is circumvented by defining
\begin{equation}
\label{eq:yawrate_actual}
r = \tilde{r} \cos\theta\,.
\end{equation}
This choice is physically justified: by \eqref{eq:x_dot} and~\eqref{eq:y_dot}, $\dot x = \dot y = 0$ whenever $\cos\theta = 0$, regardless of $\psi$. In a vertical orientation the horizontal velocity vanishes, so the yaw angle, which only sets the direction of that horizontal component, has no effect on the motion.

Since Brockett's condition~\cite{brockett1983asymptotic} rules out continuous time-invariant stabilization of~\eqref{eq:cart_3d_sys}, we take the spherical transformation of~\cite{kim2026spherical3d}, defined in Table~\ref{tab:spherical_coords} and shown in Figure~\ref{fig:3d_unicycle}, in which the transformation singularity is the mechanism that evades the obstruction. The resulting coordinate transformed system is given as

\vspace{-0.5cm}
\begin{subequations}\label{eq:spheric_3d_sys}
\begin{align}
\dot{\rho} &= v\bigl(\sin\theta\sin\zeta-\cos\theta\cos\zeta\cos\gamma\bigr)\label{eq:sph_rho}\\
\dot{\delta} &= \frac{v\cos\theta}{\rho\cos\zeta}\sin\gamma\label{eq:sph_delta}\\
\dot{\gamma} &= \frac{v\cos\theta}{\rho\cos\zeta}\sin\gamma - \tilde{r} \label{eq:sph_gamma}\\
\dot{\zeta} &= \frac{v}{\rho}\bigl(\cos\theta\sin\zeta\cos\gamma+\sin\theta\cos\zeta\bigr) \label{eq:sph_zeta}\\
\dot{\theta} &= q. \label{eq:sph_theta}
\end{align}
\end{subequations}

\begin{figure}[t]
\centering
\includegraphics[width=.9\linewidth]{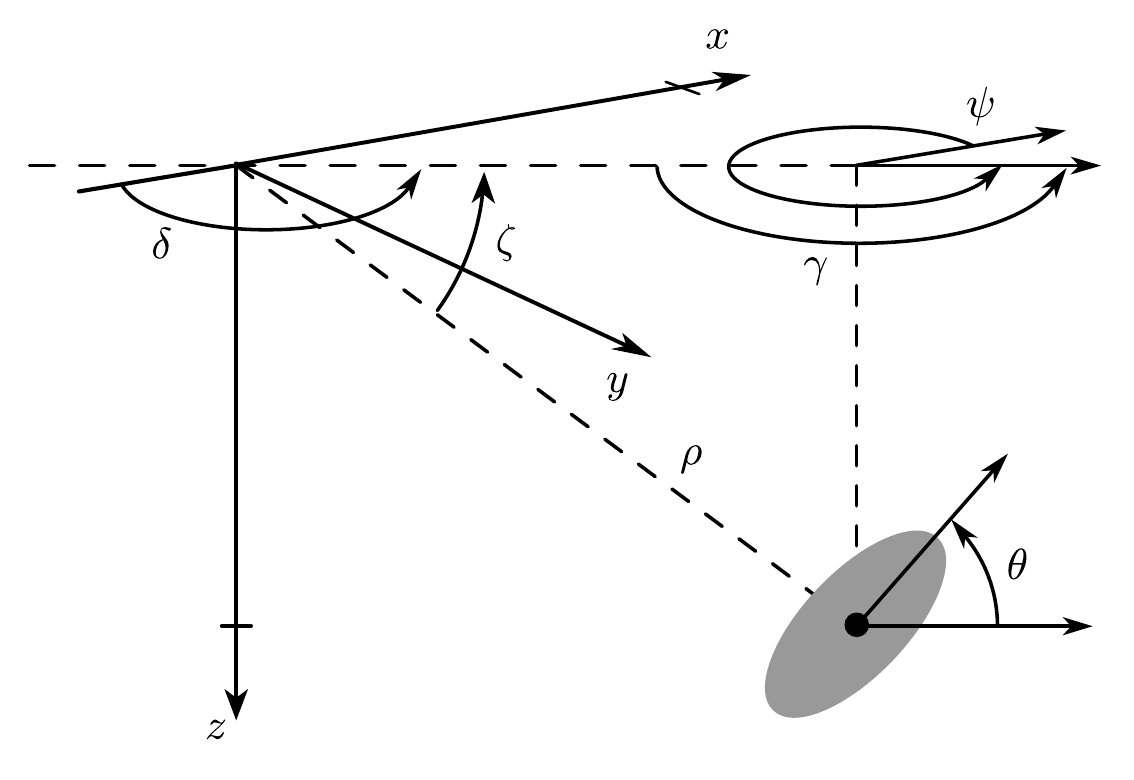}
\caption{Kinematic model of a 3D nonholonomic vehicle actuated in forward velocity (surge), pitch rate, and yaw rate.}
\label{fig:3d_unicycle}
\end{figure}

\begin{table}[t]
\centering
\renewcommand\arraystretch{1.4}
\small
\begin{tabular}{|l|l|}
\hline
\textbf{Spherical Coord.} & \textbf{Description} \\
\hline
$\rho = \sqrt{x^2 + y^2 + z^2}$ & Distance to target \\ \hline
$\delta = {\rm atan2}(-y,-x)$ & Azimuth angle \\ \hline
$\gamma = {\rm atan2}(-y,-x) - \psi$ & $z$-axis line-of-sight (LOS) angle \\ \hline
$\zeta = \arcsin\!\left(\dfrac{-z}{\rho}\right)$ & Elevation angle \\ \hline
$\theta = \theta$ & Pitch angle \\ \hline
\end{tabular}
\caption{Spherical transformation and interpretation. If the target is at ${(x^*, y^*, z^*, \psi^*) \neq 0}$, the transformation generalizes to ${\rho=\sqrt{(x-x^*)^2+(y-y^*)^2+(z-z^*)^2}}$, ${\delta= {\rm atan2}(y^*-y ,x^*-x ) -\psi^*}$, ${\gamma= \delta-\psi+\psi^*}$, and $\zeta = \arcsin\big(\frac{z^*-z}{\rho}\big)$.}
\label{tab:spherical_coords}
\end{table}

As discussed in~\cite{kim2026spherical3d}, the spherical transformation in Table~\ref{tab:spherical_coords} is undefined on the $z$-axis $\{x=y=0\}$, namely where the horizontal range $\sqrt{x^2+y^2}=\rho\cos\zeta$ vanishes and ${\rm atan2}(-y,-x)$ is undefined. The singularity appears in \eqref{eq:sph_delta}--\eqref{eq:sph_gamma} through the factor $1/(\rho\cos\zeta)$ and motivates restricting the state to the natural domain
\begin{align}
\mathcal D &\coloneqq \left\{\rho > 0\right\}\times \mathcal{T}\label{eq:domainD}\\
\mathcal{T} &\coloneqq \left\{(\delta,\gamma,\zeta,\theta)\in\mathbb{R}^4 : |\zeta|<\frac{\pi}{2}\right\}\,,
\end{align}
on which $\cos\zeta>0$ and the coordinate map is well defined. The excluded set on the $z$-axis, where the vehicle sits directly above or below the target, is a measure-zero set of codimension two whose removal leaves the state space connected. As established in~\cite{kim2026spherical3d}, no spherical or cylindrical parameterization can avoid such a ``polar'' set, and \eqref{eq:domainD} is accordingly the largest open set on which a continuous time-invariant design is possible.

Stability on $\mathcal{D}$ is measured by the metric
\begin{align}\label{eq:metricD}
|\Theta|_{\mathcal D}
\coloneqq \rho+|\delta|+|\gamma|+|\tan\zeta|+|\theta|\,,
\end{align}
where
\begin{align}
    \Theta \coloneqq (\rho,\delta,\gamma,\zeta,\theta)^\top\,.
\end{align}
The term $|\tan\zeta|$ grows unbounded as $|\zeta|\to \pi/2$ and thereby encodes the coordinate singularity at the boundary of $\mathcal D$. With this metric, we state the definitions of global asymptotic stability (GAS) and a strict CLF.

\begin{definition}[GAS on $\mathcal{D}$]
\label{def-our-GAS}
Consider the system \eqref{eq:spheric_3d_sys} with feedback laws $v,q$, and $r$
that are continuous on a state space $\mathcal{D}$ with respect to its metric. If there exists a class $\mathcal{KL}$ function $\beta$ such that, for all $t\geq t_0$, it holds that $|\Theta(t)|_{\mathcal{D}}\leq \beta\left(|\Theta(t_0)|_{\mathcal{D}},t-t_0\right)$, we say that the 
point $\Theta=0$ is {\em globally asymptotically stable (GAS) on $\mathcal{D}$}. 
\end{definition}

\begin{definition}[Strict CLF]
\label{def-CLF}
A continuously differentiable function $\Theta\mapsto V$
is a \textit{control Lyapunov function} (CLF) with respect to
\eqref{eq:spheric_3d_sys} if (i) there exist class $\mathcal{K}$
functions $(\bar\alpha_1,\bar\alpha_2)$ such that, for all $\Theta$ in
$\Sigma = \{\rho\geq 0\}\times \mathcal{T}$,
$\bar\alpha_1(|\Theta|_{\mathcal{D}}) \le V(\Theta) \le
\bar\alpha_2(|\Theta|_{\mathcal{D}})$, and (ii) for all $\Theta \neq 0$
in $\Sigma$, there exists
$\left(\dfrac{v}{\rho},q,\dfrac{r}{\cos\theta}\right)\in \mathbb{R}^3$
such that $\dot V(\Theta) < 0$.
\end{definition}

Finally, the inverse optimal redesign we pursue in this paper depends on the existence of a strict CLF. Hence, we take the following strict CLF derived as a Corollary from~\cite[Thm.~1]{kim2026spherical3d}.

\begin{corollary}[{\!\cite[Thm.~1]{kim2026spherical3d}}]
For the system~\eqref{eq:spheric_3d_sys} and for $k_1,k_2,k_3,k_4,k_5 > 0$, the function
\begin{align}\label{eq:CLF}
    V(\Theta) = \frac{1}{2}\left(k_1\rho^2 + k_2\delta^2 + k_3\tan^2\zeta + k_4e_1^2 + k_5e_2^2\right)\,,
    \end{align}
    where
    \begin{align}
    e_1 &= \gamma + \arctan(k_2\delta\cos\zeta)\label{eq:e1_def}\\
    e_2 &= \theta + \arctan\bigl(\eta(\delta,\zeta)\bigr)\,,\label{eq:e2_def}
    \end{align}
    and
    \begin{align}\label{eq:eta_def}
    \eta(\delta,\zeta)\coloneqq
    \frac{k_3\sin\zeta+\tan\zeta}{\sqrt{1+k_2^2\delta^2\cos^2\zeta}}\,,
    \end{align}
    is a strict CLF in accordance with Definition~\ref{def-CLF}.
\end{corollary}

\section{General Inverse Optimal Redesign}\label{sec:ioc_general}

With the strict CLF~\eqref{eq:CLF} in hand, we turn from stabilization alone to the design of feedback laws that are, in addition, optimal with respect to a meaningful cost. 
To state the re-design, we first recall the following transform.
 
\begin{definition}(Legendre--Fenchel transform)\label{def:legendre_Fenchel_3d}
Let $\mu$ be a class $\mathcal{K}_\infty[0,w)$ function whose derivative $\mu'$ is also a class $\mathcal{K}_\infty[0,w)$ function, where $w>0$ is finite or infinite. The mapping 
\begin{align} 
\ell\mu(s) \coloneqq \int_{0}^{s}(\mu')^{-1}(\tau)\,d\tau\,, \end{align} 
is the Legendre--Fenchel transform, where $(\mu')^{-1}$ denotes the inverse of $d\mu(s)/ds$.
\end{definition}

\vspace{-0.2cm}
\subsection{Families of inverse optimal stabilizers}
We now derive a family of inverse optimal controllers for the system~\eqref{eq:spheric_3d_sys} based on the strict CLF~\eqref{eq:CLF}.

\begin{theorem}\label{thm:IOC_3d}
Consider the system~\eqref{eq:spheric_3d_sys} on $\mathcal{D}$.
Define $\Theta\coloneqq[\rho,\delta,\gamma,\zeta,\theta]^\top$ and rewrite \eqref{eq:spheric_3d_sys} as
\begin{align}\label{eq:ioc_sys_3d}
\dot{\Theta} = \bar{g}_1(\Theta)\frac{v}{\rho} + g_2(\Theta)\,q + \bar g_3(\Theta)\frac{r}{\cos\theta}
\end{align}
with
\begin{align}
\bar{g}_1(\Theta) &\coloneqq
\begin{bmatrix}
\rho\bigl(\sin\theta\,\sin\zeta-\cos\theta\,\cos\zeta\,\cos\gamma\bigr)\\
\dfrac{\cos\theta}{\cos\zeta}\sin\gamma\\
\dfrac{\cos\theta}{\cos\zeta}\sin\gamma\\
\cos\theta\,\sin\zeta\,\cos\gamma+\sin\theta\,\cos\zeta\\
0
\end{bmatrix},
\end{align}
and
\begin{align}
g_2(\Theta)&\coloneqq
\begin{bmatrix}
0 & 0 & 0 & 0 & 1
\end{bmatrix}^{\!\top}\\
\bar g_3(\Theta)&\coloneqq
\begin{bmatrix}
0 & 0 & -1 & 0 & 0
\end{bmatrix}^{\!\top}\,.
\end{align}
Let $\mu_i\in\mathcal{K}_\infty[0,w_i)$, where $w_i > 0$ is finite or infinite and $i\in\{1,2,3\}$, and let $\varepsilon_i(\Theta)$ be continuous positive scalar-valued functions. Then, the cost functional
\begin{equation}\label{eq:J_thrm_ioc_bidir_no_roll}
J = \int_0^\infty \Biggl[l(\Theta) + \mu_1\!\left(\frac{|v|}{\varepsilon_1\rho}\right)
+ \mu_2\!\left(\frac{|q|}{\varepsilon_2}\right)
+ \mu_3\!\left(\frac{|r|}{\varepsilon_3|\cos\theta|}\right)\Biggr] {\rm d}t\,,
\end{equation}
where
\begin{align}\label{eq:run_cost_ioc_bidir_no_roll}
l(\Theta) \coloneqq \ell\mu_1(\varepsilon_1|\nu_1|)+\ell\mu_2(\varepsilon_2|\nu_2|)+\ell\mu_3(\varepsilon_3|\nu_3|)\,,
\end{align}
and $\nu_1(\Theta)\coloneqq L_{\bar{g}_1}V(\Theta)$, $\nu_2(\Theta)\coloneqq L_{g_2}V(\Theta)$, and $\nu_3(\Theta)\coloneqq L_{\bar g_3}V(\Theta)$ is evaluated along the strict CLF~\eqref{eq:CLF},
and for all initial conditions on $\mathcal{D}$, is minimized by the feedback law
\begin{subequations}\label{eq:ioc_ctrl_opt_3d}
\begin{align}
v^* &= -\rho\,\varepsilon_1 (\mu_1')^{-1}(\varepsilon_1|\nu_1|)\,\sgn(\nu_1)\\
q^* &= -\varepsilon_2 (\mu_2')^{-1}(\varepsilon_2|\nu_2|)\,\sgn(\nu_2)\\
r^* &= -\cos(\theta)\varepsilon_3 (\mu_3')^{-1}(\varepsilon_3|\nu_3|)\,\sgn(\nu_3)\,.
\end{align}
\end{subequations}
Additionally, the continuous feedback
\begin{subequations}\label{eq:ioc_ctrl_3d}
\begin{align}
v &= -\rho\,\varepsilon_1 \frac{\ell\mu_1(\varepsilon_1|\nu_1|)}{\varepsilon_1|\nu_1|}\,\sgn(\nu_1)\\
q &= -\varepsilon_2 \frac{\ell\mu_2(\varepsilon_2|\nu_2|)}{\varepsilon_2|\nu_2|}\,\sgn(\nu_2)\\
r &= -\cos(\theta)\varepsilon_3 \frac{\ell\mu_3(\varepsilon_3|\nu_3|)}{\varepsilon_3|\nu_3|}\,\sgn(\nu_3)\,,
\end{align}
\end{subequations}
is continuous in $(\nu_1,\nu_2,\nu_3)$ and renders $\Theta=0$ GAS on $\mathcal{D}$ in accordance with Definition~\ref{def-our-GAS}.
\end{theorem}

\begin{proof}
\textbf{Step 1:} \eqref{eq:ioc_ctrl_3d} is continuous in $(\nu_1,\nu_2,\nu_3)$ and stabilizes \eqref{eq:ioc_sys_3d}.
From Definition~\ref{def:legendre_Fenchel_3d} and the Leibniz rule, $(\ell\mu)'(s) = (\mu')^{-1}(s)$. Applying L'H\^opital's rule gives
\begin{align}\label{eq:lhopital_3d}
\lim_{s\to 0}\frac{\ell\mu(s)}{s}=\lim_{s\to 0}(\mu')^{-1}(s)=0\,.
\end{align}
Since $\mu'$ is class $\mathcal{K}_\infty$ and thus $(\mu')^{-1}$ is class $\mathcal{K}_\infty$ and equals $0$ at $s=0$.
Hence each component in \eqref{eq:ioc_ctrl_3d} is continuous in $\nu_i$. Next, along \eqref{eq:ioc_sys_3d},
\begin{align}
\dot V
= \nu_1\frac{v}{\rho}+\nu_2 q+\nu_3\frac{r}{\cos\theta}\,.
\end{align}
Substituting \eqref{eq:ioc_ctrl_3d} yields
\begin{align}\label{eq:Vdot_ctrl_3d}
\dot V\big|_{\eqref{eq:ioc_ctrl_3d}}
= -\ell\mu_1(\varepsilon_1|\nu_1|)
  -\ell\mu_2(\varepsilon_2|\nu_2|)
  -\ell\mu_3(\varepsilon_3|\nu_3|) < 0\,.
\end{align}
By \cite[Lemma A1]{krstic1998inverse}, $\ell\mu_i$ are class $\mathcal{K}_\infty$ functions, and since $V$ is a strict CLF, it follows by definition that $(\nu_1,\nu_2,\nu_3) = (0,0,0)$ if and only if $\Theta = 0$. Thus, \eqref{eq:Vdot_ctrl_3d} holds and $\Theta=0$ is GAS on $\mathcal{D}$ under \eqref{eq:ioc_ctrl_3d}.

\textbf{Step 2:} \eqref{eq:ioc_ctrl_opt_3d} stabilizes \eqref{eq:ioc_sys_3d}.
Substituting \eqref{eq:ioc_ctrl_opt_3d} into $\dot V=\nu_1\frac{v}{\rho}+\nu_2 q+\nu_3\frac{r}{\cos\theta}$ gives
\begin{align}\label{eq:Vdot_ctrl_opt_3d}
\dot V\big|_{\eqref{eq:ioc_ctrl_opt_3d}}
= -\sum_{i=1}^3 \varepsilon_i|\nu_i|\,(\mu_i')^{-1}(\varepsilon_i|\nu_i|).
\end{align}
Using the identity (see~\cite[Lemma A1]{krstic1998inverse}) $s(\mu')^{-1}(s)=\ell\mu(s)+\mu\big((\mu')^{-1}(s)\big)$, \eqref{eq:Vdot_ctrl_opt_3d} becomes
\begin{align}
\dot V\big|_{\eqref{eq:ioc_ctrl_opt_3d}}
&= -\sum_{i=1}^3 \ell\mu_i(\varepsilon_i|\nu_i|)
   -\sum_{i=1}^3 \mu_i\!\Big((\mu_i')^{-1}(\varepsilon_i|\nu_i|)\Big)\nonumber\\
&\le -\sum_{i=1}^3 \ell\mu_i(\varepsilon_i|\nu_i|)
= \dot V\big|_{\eqref{eq:ioc_ctrl_3d}} < 0\,.
\end{align}
Hence, $\Theta=0$ is GAS on $\mathcal{D}$ under \eqref{eq:ioc_ctrl_opt_3d}.

\textbf{Step 3:} \eqref{eq:ioc_ctrl_opt_3d} minimizes \eqref{eq:J_thrm_ioc_bidir_no_roll}.
Add and subtract $\dot V$ inside the integral and use $\lim_{t\to\infty}V(\Theta(t))=0$ to write
\begin{align}
J
=& \int_0^\infty \Bigl[l(\Theta)+\mu_1\!\left(\frac{|v|}{\varepsilon_1\rho}\right)
+\mu_2\!\left(\frac{|q|}{\varepsilon_2}\right)\\
&+\mu_3\!\left(\frac{|r|}{\varepsilon_3|\cos\theta|}\right)
+\dot V-\dot V\Bigr] {\rm d}t\nonumber\\
=& V(\Theta(0))-\lim_{t\to\infty}V(\Theta(t))\nonumber\\
&+\int_0^\infty\Biggl[l(\Theta)
+\mu_1\!\left(\frac{|v|}{\varepsilon_1\rho}\right)+\mu_2\!\left(\frac{|q|}{\varepsilon_2}\right)\nonumber\\
&+\mu_3\!\left(\frac{|r|}{\varepsilon_3|\cos\theta|}\right) +\nu_1\frac{v}{\rho}+\nu_2 q+\nu_3\frac{r}{\cos\theta}\Biggr] {\rm d}t\\
=& V(\Theta(0))\nonumber\\
&+\int_0^\infty\Biggl[
\ell\mu_1(\varepsilon_1|\nu_1|)+\mu_1\!\left(\frac{|v|}{\varepsilon_1\rho}\right)+\nu_1\frac{v}{\rho}\nonumber\\
&
+\ell\mu_2(\varepsilon_2|\nu_2|)+\mu_2\!\left(\frac{|q|}{\varepsilon_2}\right)+\nu_2 q\nonumber\\
&
+\ell\mu_3(\varepsilon_3|\nu_3|)+\mu_3\!\left(\frac{|r|}{\varepsilon_3|\cos\theta|}\right)+\nu_3\frac{r}{\cos\theta}
\Biggr] {\rm d}t.\label{eq:ioc_J_proof_3d_1}
\end{align}
By the generalized Young inequality associated with the Legendre--Fenchel transform~\cite[Thm.~156]{hardy_inequalities_1989}, we get
\begin{align}
\left(\frac{v}{\varepsilon_1\rho}\right)\left(-\varepsilon_1\nu_1\right)
&\le \mu_1\!\left(\frac{|v|}{\varepsilon_1\rho}\right)+\ell\mu_1(\varepsilon_1|\nu_1|),\\
\left(\frac{q}{\varepsilon_2}\right)\left(-\varepsilon_2\nu_2\right)
&\le \mu_2\!\left(\frac{|q|}{\varepsilon_2}\right)+\ell\mu_2(\varepsilon_2|\nu_2|),\\
\left(\frac{r}{\varepsilon_3\cos\theta}\right)\left(-\varepsilon_3\nu_3\right)
&\le \mu_3\!\left(\frac{|r|}{\varepsilon_3|\cos\theta|}\right)+\ell\mu_3(\varepsilon_3|\nu_3|)\,.
\end{align}
Multiplying by $-1$ yields
\begin{subequations}\label{eq:ioc_young_3d}
\begin{align}
\nu_1\frac{v}{\rho}
&\ge -\mu_1\!\left(\frac{|v|}{\varepsilon_1\rho}\right)-\ell\mu_1(\varepsilon_1|\nu_1|),\label{eq:ioc_young1_3d}\\
\nu_2 q
&\ge -\mu_2\!\left(\frac{|q|}{\varepsilon_2}\right)-\ell\mu_2(\varepsilon_2|\nu_2|),\label{eq:ioc_young2_3d}\\
\nu_3\frac{r}{\cos\theta}
&\ge -\mu_3\!\left(\frac{|r|}{\varepsilon_3|\cos\theta|}\right)-\ell\mu_3(\varepsilon_3|\nu_3|).\label{eq:ioc_young3_3d}
\end{align}
\end{subequations}
Therefore, \eqref{eq:ioc_J_proof_3d_1} achieves its minimum when
\eqref{eq:ioc_young_3d} holds with equality. The equality condition in \cite[Thm.~156]{hardy_inequalities_1989} implies (for each $i$)
\begin{align}
a_i = (\mu_i')^{-1}(|b_i|)\frac{b_i}{|b_i|},
\end{align}
with
\begin{align*}
a_1\coloneqq \frac{v}{\varepsilon_1\rho},\quad a_2\coloneqq \frac{q}{\varepsilon_2},\quad 
a_3\coloneqq \frac{r}{\varepsilon_3\cos\theta}\,,
\end{align*}
and $b_i = -\varepsilon_i\nu_i$, which yields
\begin{subequations}
\begin{align}
v &= -\rho\,\varepsilon_1(\mu_1')^{-1}(\varepsilon_1|\nu_1|)\sgn(\nu_1)=v^*,\\
q &= -\varepsilon_2(\mu_2')^{-1}(\varepsilon_2|\nu_2|)\sgn(\nu_2)=q^*,\\
r &= -\cos(\theta)\varepsilon_3(\mu_3')^{-1}(\varepsilon_3|\nu_3|)\sgn(\nu_3)=r^*.
\end{align}
\end{subequations}
Thus, \eqref{eq:ioc_young_3d} are equalities if and only if $(v,q,r)=(v^*,q^*,r^*)$. Substituting the equalities into
\eqref{eq:ioc_J_proof_3d_1} gives $J=V(\Theta(0))$, proving that \eqref{eq:ioc_ctrl_opt_3d} minimizes \eqref{eq:J_thrm_ioc_bidir_no_roll}.
\end{proof}

\subsection{Stability margins}

A key practical benefit of inverse optimality is robustness to actuator uncertainty, namely, the mismatch between the commanded input $u=[u_1,u_2,u_3]^\top\coloneqq[v^*,q^*,r^*]^\top$ of~\eqref{eq:ioc_ctrl_opt_3d} and the actual input. The gain margin is infinite: replacing each $u_i$ by $\kappa_iu_i$ with any constant $\kappa_i\in (0,\infty)$ still leaves $\dot V$ in~\eqref{eq:Vdot_ctrl_opt_3d} negative definite, whereas optimal controllers for systems with drift tolerate only $\kappa_i\in[\tfrac{1}{2},\infty)$~\cite{sepulchre1997constructive}, due to the system~\eqref{eq:spheric_3d_sys} being driftless (see~\cite{Part2_kimIOC2025}). Likewise, GAS on $\mathcal{D}$ is preserved when each control input is distorted by any static actuator nonlinearity $\varphi_i(u_i)$ that preserves the sign of the commanded input (i.e., $s\varphi_i(s)>0$ for $s\neq0$), a sector margin in the sense of~\cite{sepulchre1997constructive}. The classical phase margin, a frequency-domain notion, does not extend to nonlinear systems~\cite{sepulchre1997constructive}. However, since $\dot V$ equals $\nu_1\tfrac{v}{\rho}+\nu_2 q+\nu_3\tfrac{r}{\cos\theta}$, the system is \emph{lossless} (i.e., $V$ acts as an energy altered only through the inputs). Consequently, the $L_gV$-proportional members of~\eqref{eq:ioc_ctrl_opt_3d} remain stabilizing when the plant receives $a(I+\mathcal{P})u$ in place of $u$, that is, the commanded input plus a dynamic distortion $Pu$, with $P$ strictly passive in the sense of~\cite{byrnes1991passivity}, scaled by an unknown constant gain $a > 0$. Any $a>0$ is admissible, in contrast to the requirement $a\geq\tfrac{1}{2}$ for systems with drift~\cite[Prop.~2]{krstic_rigidspacecraft_1999}, \cite{sepulchre1997constructive}.


\vspace{-0.2cm}
\section{Design Flexibility via Penalty Selection}\label{sec:examples}
\vspace{-0.1cm}

\begin{figure*}[!t]
\centering
\includegraphics[width=0.27\textwidth]{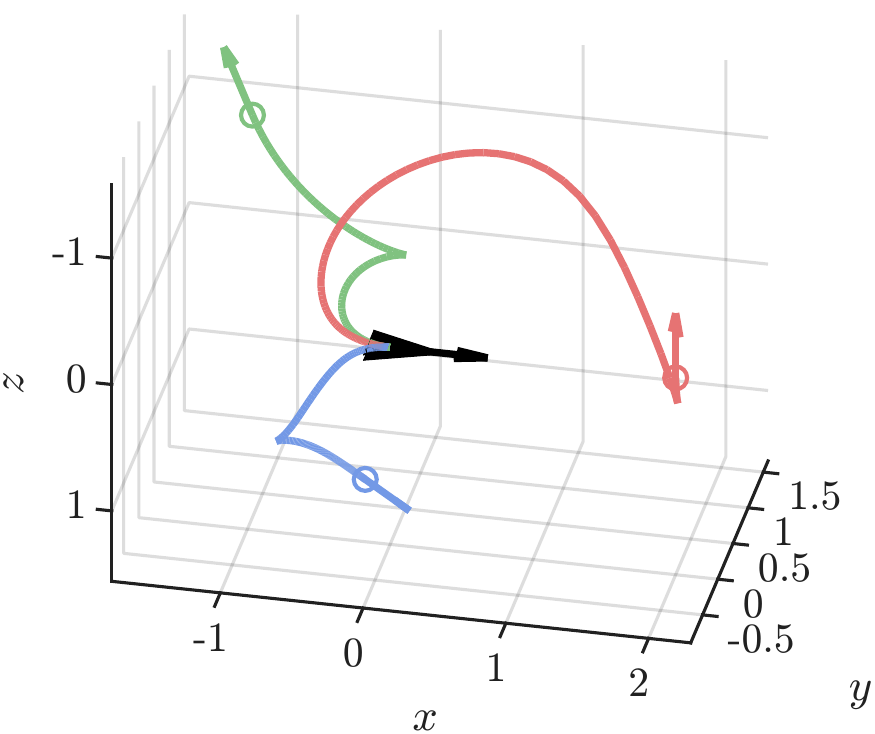}\hspace{0.02\textwidth}%
\includegraphics[width=0.27\textwidth]{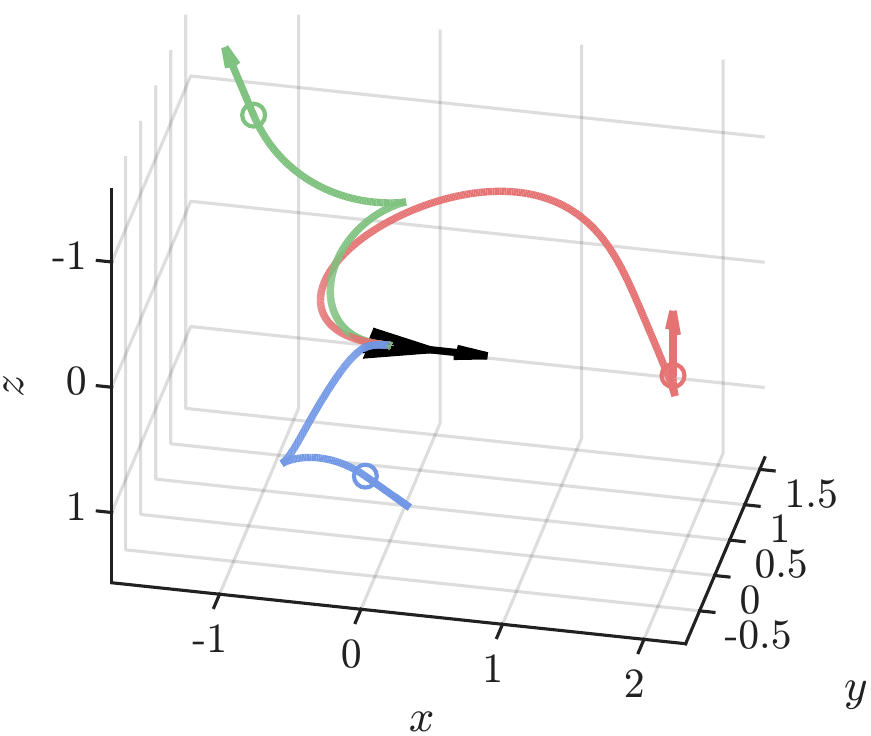}
\caption{Closed-loop 3D trajectories (NED convention) under Examples~\ref{ex:quad} (left) and~\ref{ex:lncos} (right) from three initial conditions $\Theta_1(0), \Theta_2(0),$ and $\Theta_3(0)$. The black arrow represents the target (origin).}
\label{fig:sim_traj}
\end{figure*}

\begin{figure*}[!t]
\centering
\includegraphics[width=0.345\textwidth]{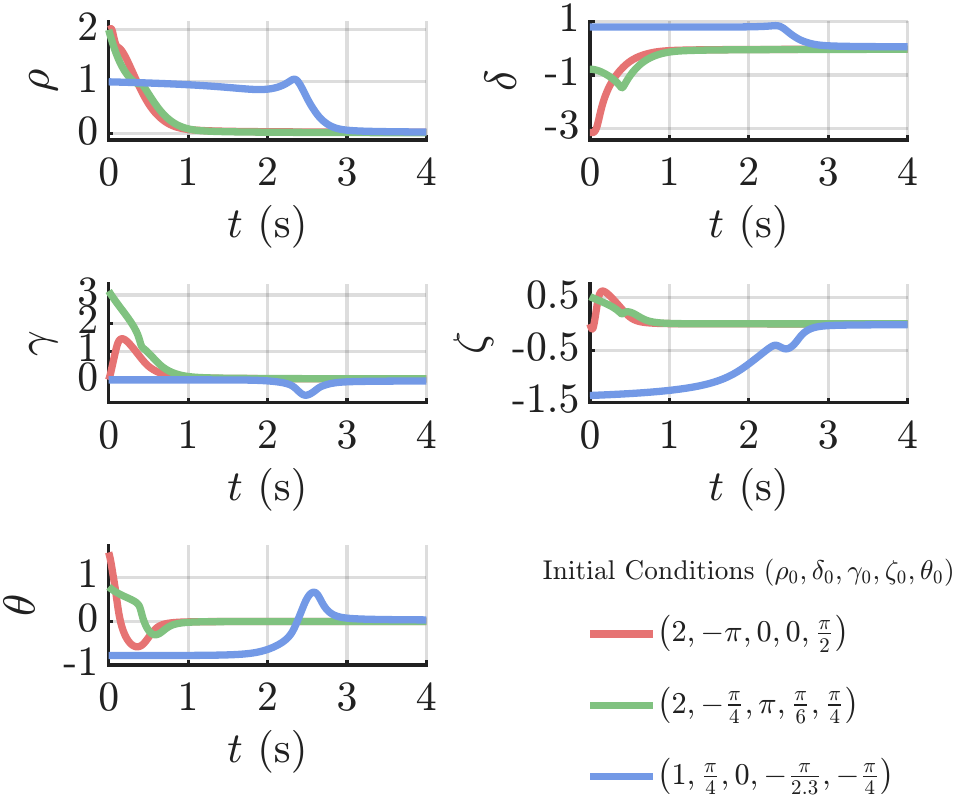}\hfill%
\includegraphics[width=0.345\textwidth]{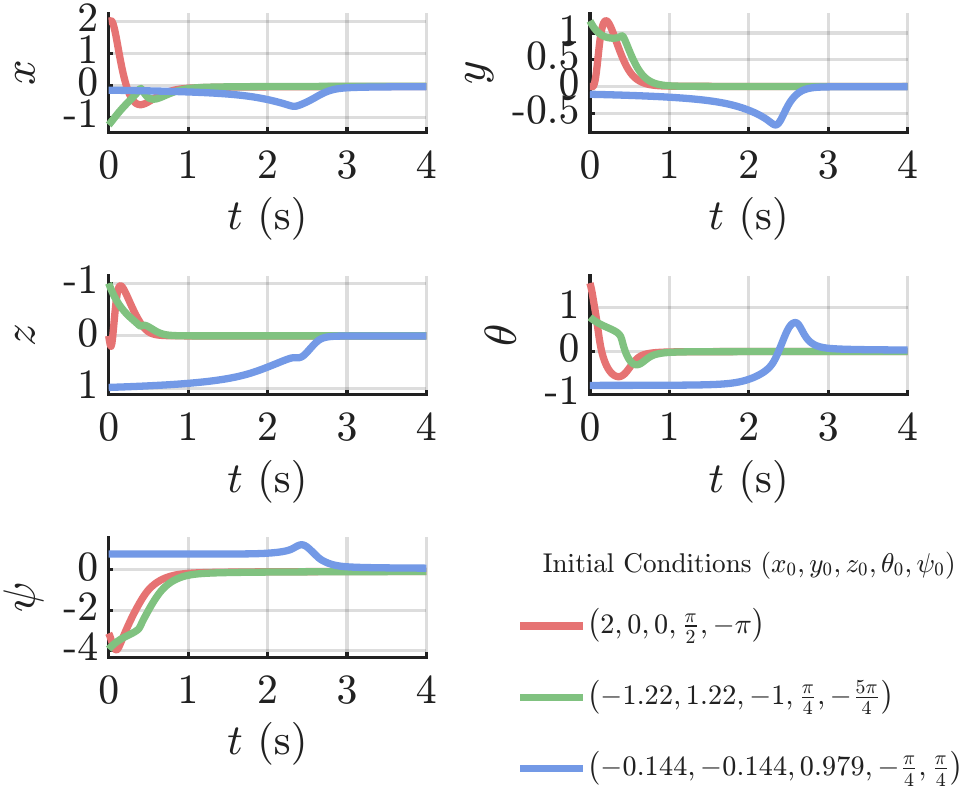}\hfill%
\includegraphics[width=0.242\textwidth]{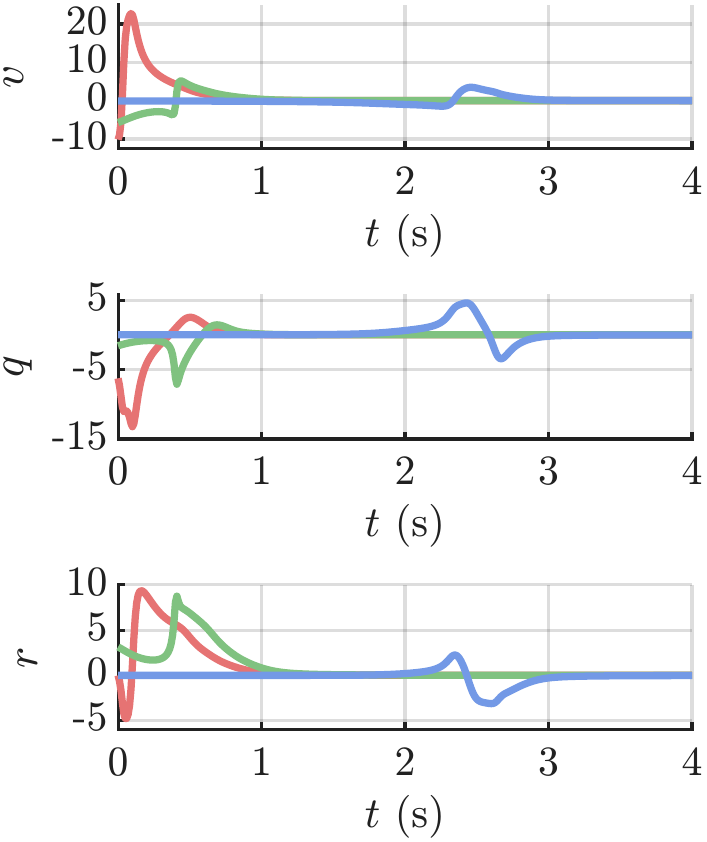}
\caption{Example~1 (quadratic penalty~\eqref{eq:mu_quad} with~\eqref{eq:eps_quad}, $Q = I_{5\times 5}$, and $\hat\varepsilon=10^{-4}$): closed-loop spherical states (left), Cartesian/Euler states (center), and optimal inputs~\eqref{eq:example1_control} (right), resulting in fast convergence with very large control effort.}
\label{fig:sim_quad}
\end{figure*}

\begin{figure*}[!t]
\centering
\includegraphics[width=0.345\textwidth]{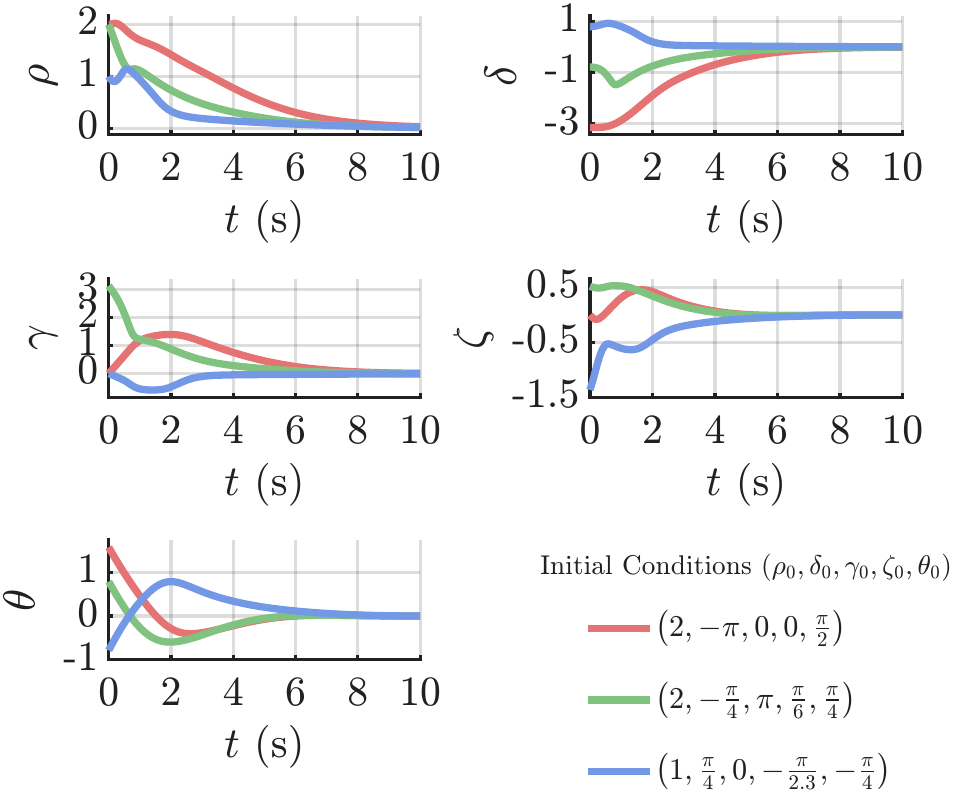}\hfill%
\includegraphics[width=0.345\textwidth]{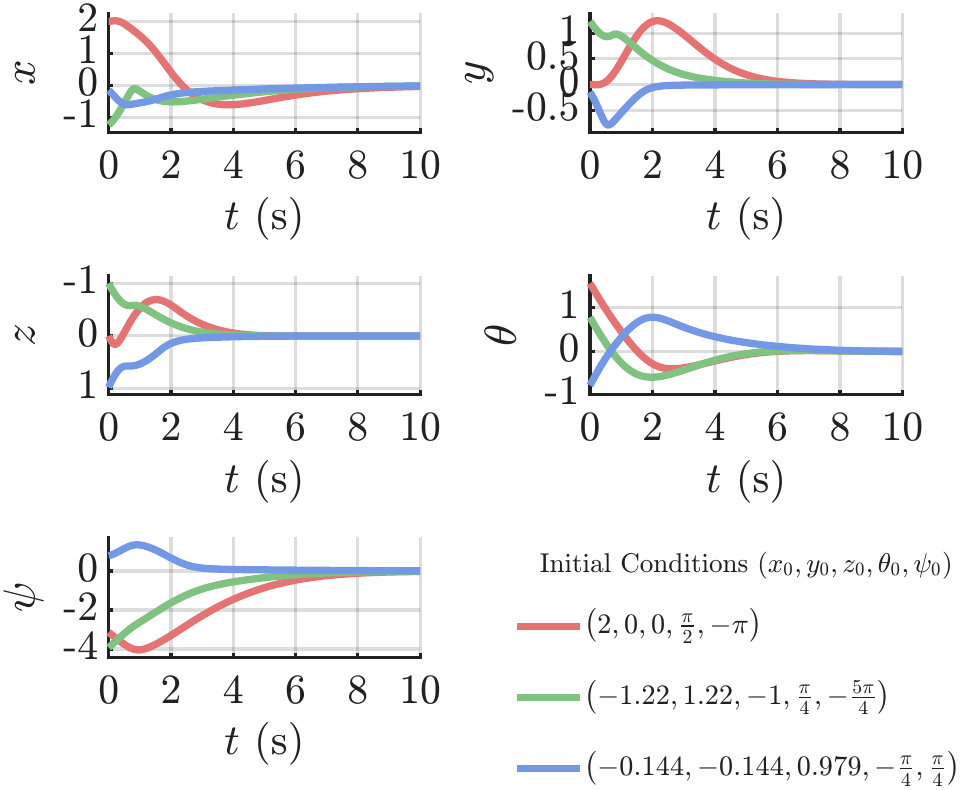}\hfill%
\includegraphics[width=0.242\textwidth]{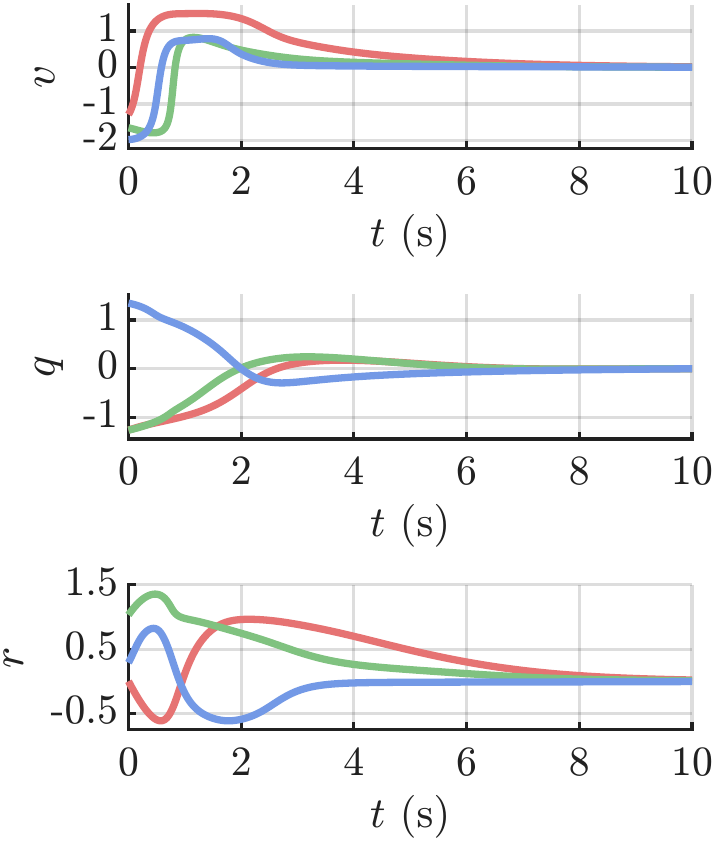}
\caption{Example~2 (input-constrained penalty, $c_i=0.5$): closed-loop spherical states (left), Cartesian/Euler states (center), and optimal inputs~\eqref{eq:example2_control} (right). All inputs remain within the bounds $|v(t)|\leq2$, $|q(t)| \leq \pi/2$, and $|r(t)| \leq \pi/2$ for all $t$.}
\label{fig:sim_lncos}
\end{figure*}

We now provide two examples to demonstrate the flexibility of Theorem~\ref{thm:IOC_3d}. Both examples below are simulated under the optimal law~\eqref{eq:ioc_ctrl_opt_3d} with CLF gains $(k_1,\dots,k_5)=(0.5,0.8,1.2,1.7,1)$, from the three initial conditions $\Theta_1(0)=\left(2,-\pi,0,0,\tfrac{\pi}{2}\right)$, $\Theta_2(0)=\left(2,-\tfrac{\pi}{4},\pi,\tfrac{\pi}{6},\tfrac{\pi}{4}\right)$, and $\Theta_3(0)=\left(1,\tfrac{\pi}{4},0,-\tfrac{\pi}{2.3},-\tfrac{\pi}{4}\right)$. While Theorem~\ref{thm:IOC_3d} permits different penalties $\mu_i$ for
each input, both examples below take $\mu_1=\mu_2=\mu_3$ for
clarity of exposition.

\begin{example}[Near classical direct optimal control]\label{ex:quad}
Consider the uniform quadratic penalty
\begin{align}\label{eq:mu_quad}
\mu_i(s) = \frac{s^2}{2}\,,
\end{align}
where $i \in \{1,2,3\}$, whose Legendre--Fenchel transform is $\ell\mu_i(s) = s^2/2$, since
$(\mu_i')^{-1}(s) = s$. The resulting optimal control law is
\vspace{-0.2cm}
\begin{subequations}\label{eq:example1_control}
    \begin{align}
    v^* &= -\rho\varepsilon_1^2 \nu_1\\
    q^* &= -\varepsilon_2^2\nu_2\\
    r^* &= -\cos(\theta)\varepsilon_3^2\nu_3 \,,
    \end{align}
\end{subequations}
which minimizes the cost functional~\eqref{eq:J_thrm_ioc_bidir_no_roll} that is of the classical quadratic form, with the running cost 
\begin{align}
    l(\Theta) = \frac{1}{2}\left(\varepsilon_1^2 \nu_1^2 + \varepsilon_2^2 \nu_2^2 + \varepsilon_3^2 \nu_3^2\right)\,.
\end{align}

Now, consider the choice
\begin{align}\label{eq:eps_quad}
\varepsilon_i = \sqrt{\frac{2\left(\Theta^\top Q\,\Theta + \hat\varepsilon\right)}{\nu_1^2 + \nu_2^2 + \nu_3^2 +\hat\varepsilon}}\,,
\end{align}
where $Q\succeq 0$ and $\hat\varepsilon > 0$. We observe that as $\hat\varepsilon \to 0$, the running cost approaches
\vspace{-0.15cm}
\begin{align}\label{eq:l_quad_limit}
l(\Theta) \to \Theta^\top Q\,\Theta\,,
\end{align}
recovering the classical
direct optimal control's quadratic in state running cost.
Fig.~\ref{fig:sim_traj} (left) and Fig.~\ref{fig:sim_quad} show the resulting closed-loop response with $Q = I_{5\times5}$ and $\hat\varepsilon = 10^{-4}$: convergence is fast at the cost of very large control effort (note the input scales in Fig.~\ref{fig:sim_quad}), which can be impractical for platforms with actuator limits and motivates Example~\ref{ex:lncos}, where the user prescribes explicit input bounds.
\end{example}

\vspace{-0.05cm}
\begin{example}[User-defined input constrained stabilization]\label{ex:lncos}
Consider the choice of
\vspace{-0.15cm}
\begin{align}
\mu_i(s) = -c_i\ln(\cos(s))\,,\quad c_i>0\,,
\end{align}
where $i\in\{1,2,3\}$, which is of class $\mathcal{K}_{\infty}[0,\pi/2)$ and $(\mu_i')^{-1}(s) = \arctan(s/c_i)$. The optimal control is then
\begin{subequations}\label{eq:example2_control}
    \begin{align}
    v^* &= -\rho\varepsilon_1\arctan\left(\frac{\varepsilon_1\nu_1}{c_1}\right)\\
    q^* &= -\varepsilon_2\arctan\left(\frac{\varepsilon_2\nu_2}{c_2}\right)\\
    r^* &= -\cos(\theta)\varepsilon_3\arctan\left(\frac{\varepsilon_3\nu_3}{c_3}\right)\,.
\end{align}
\end{subequations}

\vspace{-0.2cm}
For $|\nu_i| \ll c_i/\varepsilon_i$, the feedback is nearly linear in $\nu_i$, reducing, in the inputs $(v/\rho,\, q,\, r/\cos\theta)$, to the $L_gV$ damping law $ -(\varepsilon_i^2/c_i)\,\nu_i$, whereas as $c_i \to 0$ the feedback approaches the saturated signum law
$-\tfrac{\pi}{2}\varepsilon_i\,\mathrm{sgn}(\nu_i)$, i.e., essentially bang-bang control within the imposed bounds.

Now, impose the bounds $|v(t)| \le v_{\max} = 2$ on the surge input and $|q(t)| \le q_{\max}$, $|r(t)| \le r_{\max}$ with $q_{\max} = r_{\max} = \pi/2$ on the steering rates, for all $t$, and select
\begin{align}\label{eq:surge_epsilon_bound}
\varepsilon_1(\Theta)&=
\frac{v_{\max}}{0.01 + \rho}\frac{2}{\pi}\,,
\end{align}
and $\varepsilon_2 = \varepsilon_3 = 1$. The choice of~\eqref{eq:surge_epsilon_bound} prevents the surge input, which is proportional to $\rho$, from exceeding the imposed bound $v_{\mathrm{max}}$ for all $\rho$. Figure~\ref{fig:sim_lncos} shows the closed-loop response with $c_i=0.5$ and Fig.~\ref{fig:sim_traj} (right) the corresponding trajectories: every control input remains bounded, but at the price of visibly slower convergence than in Example~\ref{ex:quad}.
\end{example}

\section{Conclusion}

Leveraging the recently constructed strict CLFs in spherical coordinates, this paper develops a family of inverse optimal controllers for the 3D nonholonomic vehicle, each of which renders the origin GAS on the largest domain permitted by the transformation while minimizing a meaningful cost functional, without solving an HJB equation. Within this framework, the cost structure and input constraints are set by the choice of penalty function rather than by separate analysis, extending to 3D the design freedom previously available only for the unicycle. Furthermore, the quadratic choice recovers a near-classical running cost, while the bounded feedback laws resolve the unbounded control effort of the nominal design near the excluded $z$-axis set. The resulting feedback demands bidirectional surge, which some marine and aerial vehicles cannot execute. Future work will study ``cusp-free'' (forward surge only) stabilization and the extension to 3D vehicles with roll dynamics.


\makeatletter
\let\@oldthebib\thebibliography
\def\thebibliography#1{\@oldthebib{#1}\interlinepenalty=10000\relax}
\makeatother

\bibliographystyle{IEEEtranS}
\bibliography{bib,bib-Dubins,root}


\end{document}